\documentclass[12pt]{article}

\newlength{\stefan}
\usepackage{amssymb}
\usepackage{amsmath, amsthm, amsfonts, makeidx}
\usepackage[all]{xy}

\usepackage[usenames,dvipsnames]{color}

\DeclareMathSymbol{\subsetneq}{\mathord}{AMSb}{"26}

\newtheorem{lemma}{Lemma}[section]

\newtheorem{theorem}[lemma]{Theorem}

\newtheorem{proposition}[lemma]{Proposition}
\newtheorem{corollary}[lemma]{Corollary}
\theoremstyle{definition}

\newtheorem{definition}[lemma]{Definition}

\newtheorem{remark}[lemma]{Remark}
\newtheorem{conjecture}[lemma]{Conjecture}

\newcommand{\lp}{\longrightarrow}

\newcommand{\mb}{\mathbb}

\newcommand{\F}{\mb{F}}

\newcommand{\Z}{\mb{Z}}
\newcommand{\N}{\mb{N}}
\newcommand{\Q}{\mb{Q}}

\newcommand{\desda}{\Longleftrightarrow}
\newcommand{\Aff}{\mathit{Aff}}

\renewcommand{\ker}{\operatorname{ker}}

\renewcommand{\deg}{\operatorname{deg}}

\newcommand{\rad}{\operatorname{rad}}

\newcommand{\Jac}{\operatorname{Jac}}
\newcommand{\GL}{\operatorname{GL}}

\newcommand{\GA}{\operatorname{GA}}
\newcommand{\SKE}{\operatorname{SKE}}
\newcommand{\KE}{\operatorname{KE}}

\newcommand{\kar}{\operatorname{char}}

\newcommand{\MA}{\operatorname{ME}}
\newcommand{\TA}{\operatorname{TA}}
\newcommand{\ME}{\operatorname{ME}}

\newcommand{\BA}{\operatorname{BA}}

\newcommand{\SA}{\operatorname{SA}}
\newcommand{\STA}{\operatorname{STA}}

\title{A new formulation of the Jacobian Conjecture in characteristic $p$}
\author{
\begin{tabular}{ll}
Stefan Maubach & Abdul Rauf$\footnote{Supported by DAAD grant ( funding program ID 57076385).}$\\ \small Hogeschool Rotterdam & \small Jacobs University Bremen\\\small
Rotterdam, The Netherlands& \small Bremen, Germany \\ \small  s.j.maubach@hr.nl & \small a.rauf@jacobs-university.de
\end{tabular}}

\begin{document}

\maketitle

\begin{abstract}
The Jacobian Conjecture uses the equation $\det(\Jac(F))\in k^*$, which is a very short way to write down many equations putting restrictions on the coefficients of a polynomial map $F$. In characteristic $p$ these equations do not suffice to (conjecturally) force a polynomial map to be invertible. In this article, we describe how to construct the  conjecturally sufficient equations in characteristic $p$ forcing a polynomial map to be invertible.
This provides an (alternative to  Adjamagbo's formulation)  definition of the Jacobian Conjecture in characteristic $p$.
We strengthen this formulation by investigating some special cases and by linking it to the regular Jacobian Conjecture in characteristic zero.
\end{abstract}



\section{Introduction}

\subsection{Notations and definitions}
All rings are commutative with 1.
We denote $R^{[n]}=R[x_1,\ldots, x_n]$ the polynomial ring in $n$ variables over a ring $R$.
A polynomial map (or polynomial endomorphism) $F$  with coefficients in a ring $R$ is a list of polynomials $(F_1,\ldots, F_n)$ where $F_i\in R^{[n]}$.
Such a polynomial map provides an endomorphism of $R^{[n]}$ as well as a map $R^n\lp R^n$. Since $R$ can be a finite field/ring, we cannot identify these viewpoints. (A polynomial map can induce the identity map $R^n\lp R^n$ while not being the identity endomorphism.)

We define $\MA_n(R)$ as the set of polynomial endomorphisms on $R^{[n]}$. This forms a monoid w.r.t. composition, and the subset of invertible elements in this monoid is denoted by $\GA_n(R)$ and is the group of polynomial automorphisms. We define $\deg(F)=\max(\deg(F_1),\ldots, \deg(F_n))$ for $F\in \MA_n(R)$.
The set of affine automorphisms $\Aff_n(R)$ is $\{F\in \GA_n(R) ~|~\deg(F)=1\}$. A polynomial map $F\in \MA_n(R)$ is triangular if $F_i\in R[x_i,\ldots, x_n]$.
If $F$ is a triangular automorphism and $R$ is a domain, it turns out to be of the form $(r_1x_1+f_1,\ldots, r_nx_n+f_n)$ where $r_i\in R^*$ and $f_i\in k[x_{i+1},\ldots, x_n]$. The set of triangular automorphisms is denoted by $\BA_n(R)$. Both $\BA_n(R)$ and $\Aff_n(R)$ turn out to be subgroups of $\GA_n(R)$.
We define $\TA_n(R):=<\BA_n(R), \Aff_n(R)>$, the tame automorphism group.

We define $\SA_n(R)=\{ F\in \GA_n(R) ~|~\det(\Jac(F))=1\}$. Similarly, we define $\STA_n(R)=\SA_n(R)\cap \TA_n(R)$  etc.

For each of these sets, we define $\MA^d_n(R)=\{F\in \MA_n(R) ~|~\deg(F)\leq d\}$, $\GA^d(R)=\GA_n(R)\cap \MA_n^d(R)$ etc.

We use the notation $x^{\alpha}=x_1^{\alpha_1}\cdots x_n^{\alpha_n}$ if $\alpha\in \N^n$.

\subsection{The Jacobian Conjecture}

The Jacobian Conjecture is a quite notorious conjecture in the field of Affine Algebraic Geometry. One formulation is:\\

{\bf (JC(R,n))}: If $F\in \MA_n(R)$ where $R$ is a domain of characteristic zero, then $\det(\Jac(F))\in R^*$ implies that $F\in \GA_n(R)$. \\

For many details we can refer to the book \cite{E00}. The conjecture is widely open even in the case $n=2$ (and trivial in dimension 1). Proving $JC(R,n)$ for any $R$ of characteristic zero yields $JC(R,n)$ true for all domains $R$ of characteristic zero.

Naievely translating the Jacobian Conjecture into characteristic $p$ yields counterexamples, already in dimension 1 even:  the map $x-x^p$ is not injective but has $\det(\Jac(x-x^p))=1$. Therefore, Adjamagbo defined in \cite{ADA92} a possible version of the Jacobian Conjecture for fields $k$ with characteristic $\kar(k)=p$:\\

{\bf (AJC(n,p))}: Let $F=(F_1,\ldots, F_n)$ where $F_i\in k[x_1,\ldots, x_n]$ and $k$ a field of characteristic $p$. Assume that $\det(\Jac(F))\in k^*$ and additionally assume that  $p$ does not divide $[ k(x_1,\ldots,x_n): k(F_1,\ldots, F_n)]$.
Then $F$ has a polynomial inverse. \\

The ``$\kar(k)$ does not divide $[ k(x_1,\ldots,x_n): k(F_1,\ldots, F_n)]$'' requirement seems to exclude all pathological counterexamples to the Jacobian Conjecture, but adds another difficult requirement to the (deceptively simple looking but) difficult equation $\det(\Jac(F))\in k^*$. Adjamagbo showed that knowing $AJC(n,p)$ for all $p$ implies $JC(n,k)$ for all $k$.

We approach the JC in characteristic $p$ from  a different perspective: let us write down a generic polynomial automorphism of degree 2, having affine part identity:
\[ F=(x+a_1x^2+a_2xy+a_3y^2, y+b_1x^2+b_2xy+b_3y^2)\]
Then, in characteristic zero, the equation $1=\det(\Jac(F))$ yields several equations on the coefficients:
\[ \begin{array}{rl}
1=&\det(Jac(F))\\
=&1+\\
&(2a_1+b_2)x+\\
&(a_2+2b_3)y+ \\
&(2a_1b_2+2a_2b_1)x^2+\\
&(2b_2a_2+4a_1b_3+4a_3b_1)xy+\\
&(2a_2b_3+2a_3b_2)y^2\\
\end{array} \]
Then apparently, the equations $2a_1+b_2=0,$ $a_2+2b_3=0$, etc. are exactly the equations one needs to ensure that $F$ is invertible in characteristic zero.
However, in characteristic 2 the above equations are not enough to conclude that $F$ is invertible (in fact, some equations completely vanish), as an example $(x+x^2,y)$ shows.
Therefore, one needs extra equations in characteristic $p$. In fact, thinking a little deeper, we {\em know} that such equations must exist.  (Without going into detail, the equations must be the closure of the
 set of automorphisms having determinant Jacobian 1 inside $\MA_n^d(k)$, where closure is in a natural topology \cite{Shafarevich1,Shafarevich2, Stampfli})
We only have to {\em find} them.
In this article we  claim that we have found them (at least conjecturally). In fact, what we are doing is refining the regular Jacobian Conjecture so that it makes sense in characteristic $p$ also.

We make a remark on Adjamagbo's formulation w.r.t. the above considerations: note that if there exists at least one counterexample $F$ to the Jacobian Conjecture in characteristic zero, then $[k(x_1,\ldots,x_n):k(F_1,\ldots,F_n)]=d>1$. It might very well be that $F\mod{p}$ is an interesting map for any prime $p$. But, if $p|d$, then Adjamagbo's formulation excludes this example, while one could argue that a formulation of the JC  in characteristic $p$ should not.
One could say that in this case $p\nmid  [k(x_1,\ldots,x_n):k(F_1,\ldots,F_n)]$ adds {\em too many} equations, or perhaps the {\em wrong} equations.

\section{Initial considerations}

Let us consider the degree 2 example of the previous section.
One of the equations is $2a_1b_2+2a_2b_2$. In characteristic zero, this implies the equation $a_1b_2+a_2b_1$. Looking at it like this, it seems strange to exclude this latter equation in characteristic 2. Also, if we define the ideal
\[ I=(2a_1+b_2,a_2+2b_3,2a_1b_2+2a_2b_1,2b_2a_2+4a_1b_3+4a_3b_1, 2a_2b_3+2a_3b_2) \]
in the ring $\Q[a_1a_2,a_3, b_1,b_2,b_3]$, then any invertible polynomial map of degree 2 over $\Q$ will have coefficients which satisfy
{\em any} equation of $I$. Even more, they satisfy any equation appearing in $\rad(I)$.
Again, in the same vein as before, we can argue that any equation appearing in $\rad(I)$, should appear as equation in characteristic
$p$, also. Hence, in this way we can give some universal equations which should be the equations which work in any characteristic.

This is essentially the formulation of the Jacobian Conjecture in characteristic $p$ we introduce in the next section, but we have to use more formal language.

\section{A new formulation of the Jacobian Conjecture in characteristic $p$}

Given $F=(F_1,\ldots,F_n)\in \MA_n(R)$ where $R$ is some ring, we can write $F_i=\sum_{\alpha\in \N^n} c_{i, \alpha}x^{\alpha}$.
We can also make the infinitely generated ring $C_R:=R[c_{i,\alpha} ~|~1\leq i\leq n, \alpha\in \N^n]$
where the $c_{i,\alpha}$ are variables. One can now make the universal polynomial map of degree $d\in \N$ by taking the polynomial map in $ \MA_n(C_R)$ which has coefficients the variables $c_{i,\alpha}$. Lets say that $C_{R,d}$ is the finitely generated ring generated by the coefficients up to and including degree $d$. (I.e. $C_R$ is the union, or direct limit, of the rings $\ldots \subset C_{R,d}\subset C_{R,d+1}\subset \ldots$. )

\begin{definition} Let $F[d]\in \MA_n(C_{\Q,d})$ be the universal polynomial endomorphism of degree $d$ having affine part identity.
Computing $\det\Jac(F[d])-1=\sum_{\alpha\in \N^n} E_{\alpha} x^{\alpha}$ yields a polynomial in $x_1,\ldots,x_n$ having coefficients $E_{\alpha}\in C_{\Q,d}$. Define the ideal  $I_{\Q}^d=(E_{\alpha} ~|~\alpha\in \N^n)$ in $C_{\Q,d}$ generated by the equations found in the formula $\det\Jac(F[d])=1$.
We define the ideal $I_{\Q}$ in $C_{\Q}$  as the inverse limit of the canonical chain $\ldots\lp I_{\Q}^{d+1}\lp I_{\Q}^d\lp \ldots$.
It coincides with the equations found in the coefficients of $\det\Jac(F[\infty])=1$, where $F[\infty]$ is the power series with universal coefficients.
\end{definition}
\begin{definition} We define $J_{\Z}:=\rad(I_{\Q})\cap C_{\Z}$ as the ``ideal of integer keller equations''.
This ideal captures the universal equations described in the previous section.
If $R$ is a ring, we define $J_R:=J_{\Z}\otimes R$ as an ideal in the ring $C_R$.  It is the ideal in $C_R$
generated by those same equations (in characteristic zero) or by those equations modulo $p$ (in characteristic $p$). In particular, we
 define $J_p:=J_{\F_p}=J_{\Z} \mod p$ as an ideal in $C_{\F_p}$. Similarly we define $J_{\Z}^{d}:=\rad(I_{\Q}^d)\cap C_{\Z,d},$ $J_R^d:=J_{\Z}^d\otimes R$ and $J_p^d:=J_{\Z}^d \mod p$ where $d=deg(F).$

Let $N_d$ be the number of variables in $F[d]$ (i.e. the dimension of the ring $C_{R,d}$).
We say that $v\in R^{N_d}$ satisfies $J_R^d$ if $f(v)=0$ for all $f\in J_R^d$.
We say that ``$v\in R^{N_d}$ satisfies $J_R$'' if $v\in R^{N_d}$ satisfies $J_R^d$.

We can identify $F\in \MA_n(R)$ by the vector of coefficients $v(F)$ of $F$; in particular, if $F\in \MA_n(R)$ we say that ``$F$ satisfies $J_R$ ($J_R^d$)'' if $v(F)$ satisfies $J_R$ ($J_R^d$).
\end{definition}
Throughout, we will write the elements of $J_{R}=J_{\Z}\otimes R$ by $\sum_i{e_i}h_i$ instead of $\sum_i{e_i}\otimes h_i$, where $e_i\in J_{\Z}$ and $h_i\in R$ for all $i$ (for simplicity, we will omit the tensor notation in this manuscript).

\begin{definition} We say that $F\in \MA_n(R)$ is a {\bf strong Keller map} if $F$ satisfies $J_R$ (or equivalently $F\in \MA_n(R)$ is a {\bf strong Keller map} if $F$ satisfies $J_R^d$ where $deg(F)=d$).
We denote the set of strong Keller maps by $\SKE_n(R)$ and the set of Keller maps by $\KE_n(R).$
\end{definition}

\begin{conjecture}\textbf{Jacobian conjecture over any field (in particular, in positive characteristics)}
({$\bf \mathcal{JC}(k,n)$})\\ Let $k$ be a field (of characteristic $p$) and $F\in \MA_n(k)$ be a strong Keller map. Then $F\in \GA_n(k)$.
\end{conjecture}

So, an alternative definition of $\mathcal{JC}(k,n)$ is  ``$\SKE_n(k)=\GA_n(k)$''. We will use curly letter notation $\mathcal{JC}$ to represent Jacobian conjecture in characteristic $p$.

Of course, we still need to show that the above formulation coincides with the regular formulation in case the field is of characteristic zero. However, this will follow directly from lemma \ref{L3.8}.

Note that we only defined the conjecture for fields of any characteristic, but with a slight modification one can define it for all domains (of any characteristic). However, we will stick with this formulation in this first encounter.

Before we will study the validity of this conjecture, we will introduce some facts and concepts we will use afterwards.

\section{Basic facts}

Some basic facts are mentioned in the following remark about the map $F\mod p$ for $F\in\MA_n(\Z).$ They are used in various places without mentioning.
\begin{remark} Let $F\in \MA_n(\Z)$. Then
\[ (\det\Jac(F))\mod p=\det(\Jac(F)\mod p) = \det\Jac(F \mod p)\] . In particular:
\begin{itemize}
\item $\det\Jac(F)=1\mod p \desda \det\Jac(F\mod p)=1\mod p$.
\item If $F\in\MA_n(\Z)$ such that $F\mod p\in\SKE_n(\F_p)$, then $\det\Jac(F)=1+pH$ for some $H\in\MA_n(\Z)$.
\item $(F\circ G)\mod p = (F\mod p)\circ (G\mod p)$, and $\det\Jac(F\circ G)\mod p=\det\Jac (F\mod p\circ G\mod p).$
\end{itemize}
\end{remark}

\begin{proof} Writing out the equations $\det(\frac{\partial (F_i\mod p)}{\partial x_j})$ we see that checking the remark essentially comes down to checking that if $c_{\alpha} x^{\alpha}$ is a generic monomial where $\alpha \in \N^n$ and $c_{\alpha}\in \Z$, then
\[ \frac{\partial c_{\alpha} x^{\alpha}}
{\partial x_i}
\mod p =
\frac{\partial c_{\alpha} x^{\alpha}\mod p}
{\partial x_i}
\]
which is true (just check the case where $p$ divides $c_{\alpha}$ or $p$ divides $\alpha_i$ separately).
\end{proof}
The interesting thing is that it is hard to grasp of some $F\in \MA_n(\Z)$ what conditions imply $F\mod p \in \SKE_n(\F_p)$, which is stronger than just $\det\Jac(F))=1 \mod p$.

The following lemma will be used several times.
\begin{lemma}\label{ideals}
 Let $F\in\ME_n(R).$ If $F$ satisfies the ideal $I_{\Q}$ then it satisfies the ideal $J_{\Z}.$
\end{lemma}
\begin{proof}
 Consider $Q\in J_{\Z}.$ Note that $J_{\Z}=\rad(I_{\Q})\cap C_{\Z} =\rad(I_{\Q}\cap C_{\Z})$ thus there exist $n\in\Z$ such that $Q^n\in I_{\Q}\cap C_{\Z}\subset C_{\Z}.$ Assume $I_{\Q}$ is generated by $\{e_i\}_{i\in\Omega}$, then $Q^n=\sum_{i}h_ie_i$ for $h_i\in C_{\Z}$ for all $i.$ Thus $Q^n(\nu(F))=\sum_{i}h_i(\nu(F))e_i(\nu(F))=0$ since $F$ satisfies $e_i\in I_{\Q}$. Hence $Q(\nu(F))=0$ as $C_{\Z}$ is integral domain.

\end{proof}
\begin{lemma} \label{L3.8} Let $R$ be a ring with $\kar(R)=0$, then $F\in\SKE_n(R)$ if and only if $F\in\KE_n(R).$
\end{lemma}

\begin{proof} Let $F\in SKE_n(R),$ then $\forall Q\in J_R$ we have $Q(\nu(F))=0$. Since $f\otimes1\in J_R$ for $f\in J_{\Z}$ and $1\in R$ thus $f(\nu(F))=0$ for all $f\in J_{\Z}.$
As $I_{\Q}\cap C_{\Z}\subseteq J_{\Z},$ thus $f(\nu(F))=0$ for all $f\in I_{\Q}\cap C_{\Z}.$ Since for any $e\in I_{\Q}$ we can find $f\in I_{\Q}\cap C_{\Z}$ such that $e=\frac{f}{m}$ for some $m\in\Z.$ Thus $e(\nu(F))=0$ for all $e\in I_{\Q}.$ Hence $F\in\KE_n(R).$ Conversely suppose that $F$ is a Keller map, then $F$ satisfies $I_{\Q}.$ Thus by lemma \ref{ideals} we have $e(\nu(F))=0$ for all $e\in J_{\Z}.$
  Now for any $\sum_{i}e_ir_i\in J_R$ we have $\sum_{i}e_ir_i(\nu(F))=\sum_{i}e_i(\nu(F))r_i(\nu(F))=0,$ where $r_i\in R$ and $e_i\in J_{\Z}$ for all $i.$ Thus $F$ is strong Keller map.
 \end{proof}
\begin{lemma}\label{rem1}
$\SKE_n(k)\subset \SKE_n(\acute{k})$ for any fields $k\subset \acute{k}$ of positive characteristics.
\end{lemma}
\begin{proof}

Let $F\in\SKE_n(k).$ Consider $k_0$ be the subfield of $k$ generated by the coefficients of $F$ then $F\in\SKE_n(k_0)$ and so $F$ satisfies the ideal $J_{k_0}.$ Since $J_{k_0}\subset J_{\acute{k}},$ it is obvious to see  $q(\nu(F))=0$ for any $q\in {J_{\acute{k}}\setminus J_{k_0}}$ (as $q$ does not involve any coefficient of $F$ by definition of $J_{\acute{k}}$). Thus $F$ satisfies the ideal $J_{\acute{k}}$ and so $F\in \SKE_n(\acute{k}).$
\end{proof}

\section{Two surjectivity conjectures}

Given $F\in \GA_n(\Z)$ we can define $F\mod p$ for any prime $p$, yielding an element of $\GA_n(\F_p)$. If we additionally assume that $p\not | \det(\Jac(F))$ then $F\mod{p} \in \GA_n(\F_p)$, even. This yields the natural map $\pi: \SA_n(\Z)\lp \SA_n(\F_p)$. The following fact is not that difficult to prove:

\begin{remark}
$\pi(\STA_n(\Z))=\STA_n(\F_p)$.
\end{remark}

The reason for this is that (1) any affine or triangular map having determinant Jacobian 1 has a preimage under $\pi$, (2) any tame automorphism of determinant Jacobian 1 can indeed be written as a composition of affine and triangular automorphisms of determinant Jacobian 1. (See  \cite{MR14} lemma 3.4.)

Now an obvious question is whether the map $\pi : \SA_n(\Z)\lp \SA_n(\F_p)$ is surjective or not; this question is interesting as nonsurjectivity would yield non-tame maps due to the above remark. This is part of the topic of the papers \cite{MR14, M03, MW11}.
\begin{definition}\label{def2}
Let $R$ be a $\Z$ algebra and $k$ be a field such that we have a  surjective ring homomorphism $R\lp k$. We can extend it naturally from polynomial maps over $R$ to polynomial maps over $k.$ We denote this extended map by $\pi.$
\end{definition}
We notice that corresponding to each automorphism $F\in\SA_n(\Z)$ we have $F\mod p\in\SA_n(\F_p),$ but there may exist some automorphisms  $f\in\SA_n(\F_p)$ such that $\pi^{-1}(f)\notin\SA_n(\Z).$
We conjecture the following for a $\Z$ algebra $R$ and a field $k$:
\begin{conjecture} \label{conj1} Let $R$ be a $\Z$ algebra and $k$ be any field. If we have a surjective ring homomorphism $R\lp k,$ then we have
\begin{enumerate}
\item $\pi(\SA_n(R))=\SA_n(k)$.
\item $\pi^{-1}(\SA_n(k))\cap\KE_n(R)=\SA_n(R).$
\end{enumerate}
\end{conjecture}

A similar conjecture is the following (see also lemma \ref{subset1} and corollary \ref{corsub1}):

\begin{conjecture} \label{conj2} Let $R$ be a $\Z$ algebra and $k$ be any field of characteristic $p$. If we have a surjective ring homomorphism $R\lp k,$ then the map $\pi: \KE_n(R)\lp \MA_n(k)$ has $\SKE_n(k)$ in its image.
\end{conjecture}

If the above conjecture is {\em not} true, then it can mean various things: it could mean that $\mathcal{JC}(k,n)$ is not true (or should be reformulated), or that there exist non-tame automorphisms over $k$.

Assuming  $\mathcal{JC}(k,n)$ to be true, then conjecture \ref{conj1} imples \ref{conj2}, but no other implications can be made, nor does $\mathcal{JC}(k,n)$ imply any of the above conjectures.\\

{\bf Justification of the above conjectures:} The above conjectures are not made to ``match exactly what we need in our proofs''. They capture the essence of whether characteristic $p$ is {\em truly} different from characteristic zero. If one or more of these conjectures is wrong, then characteristic $p$ is in its core different from characteristic zero (for example, there might exist $\F_p$-automorphisms of $\F_p^{[n]}$ which are of a completely different nature than one can find in characteristic zero), while if both of them are correct, then characteristic $p$ is not too unsimilar from characteristic zero and both are intricately linked.

The tendency is to believe the conjectures (hence the name ``conjecture'' and not ``problem'' or ``question''): it would be really surprising if any counterexamples would not be easily constructable in low degree and dimension (and known), whereas it can be easily imagined that the conjectures are true but hard to prove. For example, due to the  the fact that we do not even have a (parametrized) list of generators for the automorphism group $\GA_n(k)$ (unlike for $\GL_n(k), \TA_n(k)$), we can understand that conjecture \ref{conj2}, if true,  is very hard to prove. \footnote{Note, that with a little change,  {\em some people} would agree on the same text for the Jacobian Conjecture. }

\section{Some computations indicating the correctness of conjecture $\mathcal{JC}(k,n)$}

We should check this conjecture for some nontrivial cases, in order to point out that it might do what it claims.
Therefore, in  this subsection we considered  polynomial endomorphisms of degree $\leq3$ with coefficients in field of characteristic $p$, and having affine part identity.
We will check if $\mathcal{JC}(k,2)$ is true for these maps for fields $k$ of characteristic $p$. Let us write down such a polynomial map with generic coefficients:
\[T=(x,y)+(Ax^2+By^2+Cxy+Dx^3+Ey^3+Fx^2y+Gxy^2,\]
\[A_1x^2+B_1y^2+C_1xy+D_1x^3+E_1y^3+F_1x^2y+G_1xy^2).\]
Let us take the determinant of the Jacobian and equal it to 1:
\[1=\det(\Jac(T))=1+(C_1+2A)x+(2B_1+C)y\]\[+(F_1+3D+2AC_1-2A_1C)x^2+(2G_1+2F+4AB_1-4A_1B)xy\]\[+(3E_1+G+2CB_1-2BC_1)y^2\]\[(6AE_1-6A_1E+4B_1F-4BF_1+CG_1-C_1G)xy^2\]\[(6DB_1-6D_1B+4AG_1-4A_1G+FC_1-F_1C)x^2y\]
This gives us generators of the ideal $I_{\Q}=(C_1+2A, 2B_1+C,\ldots)$ in the ring $\Q[A,B,\ldots, E_1]$.
It is clear that the following equations are in $I_{\Q}$ also, by doing some elementary manipulations:
   \[F_1+3D, AC_1-A_1C, G_1+F, AB_1-A_1B, 3E_1+G, CB_1-BC_1, AE_1-A_1E, B_1F-BF_1,\]\[ CG_1-C_1G, DB_1-D_1B, AG_1-A_1G, FC_1-F_1C, DE_1-D_1E,FG_1-F_1G, FA_1-F_1A,\]\[DC_1-D_1C, CE_1-C_1E, B_1G-BG_1, DG_1-GD_1, FE_1-EF_1, DF_1-D_1F\]
\begin{eqnarray}\label{equ}
C_1+2A,C+2B_1,GE_1-EG_1 \in I_{\Q}.\end{eqnarray}
Moreover it can be checked by any computer algebra package (we used singular) that
 \begin{eqnarray}\label{radi}
 A^3{E_1}^2-B^3{D_1}^2,A^3{E}^2-B^3{D}^2\in \rad(I_{\Q}),
 \end{eqnarray}
 where these equations do not belong to $I_{\Q}.$\\

As before, we define $J_{\Z}:=\rad(I_{\Q})\cap \Z[A,B,\ldots, E_1]$, and $J_p:=J_{\Z}\mod p$.
It is possible to now use a computer algebra system to show that
 \ref{equ} and \ref{radi} generate $J_p$, but this can be quite a strain on the computer system, which we can avoid in this case:
We will show that (Part 1) assuming these equations forces $T$ to be invertible for any $p$, and (Part 2) if $T$ is assumed to be invertible, then it satisfies the equations \ref{equ} and \ref{radi} (meaning we show that these equations might not generate $J_p$, but the radical of the ideal generated by them does). \\

{\bf Part 1: assuming the equations yields invertibility.}\\
 We first assume that $A, A_1, E$ are all nonzero.
 Then solving (1) and (2) yields
\[C_1=-2A,C=-\frac{2A^2}{A_1},B_1=\frac{A^2}{A_1},B=\frac{A^3}{{A_1}^2},\]\[G=-\frac{3A_1E}{A},G_1=-\frac{3{A_1}^2E}{A^2},\]\[D=-\frac{{A_1}^3E}{A^3},D_1=\frac{{A_1}^4E}{A^4}\]\[F_1=-3D,G=-3E_1,F=-G_1.\]
Thus
\[T=(x,y)+(Ax^2+\frac{A^3}{{A_1}^2}y^2-\frac{2A^2}{A_1}xy-\frac{{A_1}^3E}{A^3}x^3+Ey^3+\frac{3{A_1}^2E}{A^2}x^2y-\frac{3A_1E}{A}xy^2,\]
\[A_1x^2+\frac{A^2}{A_1}y^2-2Axy-\frac{{A_1}^4E}{A^4}x^3+\frac{A_1E}{A}y^3+\frac{3{A_1}^3E}{A^3}x^2y-\frac{3{A_1}^2E}{A^2}xy^2).\]
This can be rewritten as
\[ T=
\left(
\begin{matrix}
x+A(x-\frac{A}{A_1}y)^2-E\frac{A^3}{A_1^3}(x-\frac{A}{A_1}y)^3, \\
y+  A_1(x^2-\frac{A}{A_1}y)^2-\frac{A_1^4E}{A^4}(x-\frac{A}{A_1}y)^3
\end{matrix} \right)   \]
Regardless of characteristics, $T$ is a tame map of the form
\[ T=(x+\frac{A}{A_1}y, y) (x,y+A_1x^2-\frac{EA_1^3}{A^3}x^3)(x-\frac{A}{A_1}y,y) \]
meaning that $T$ is invertible.\\
In case that  one or more of $A,A_1,E$ are zero are easier than the above case (many coefficients are forced to be zero in these cases) and we leave it to the reader.

{\bf Part 2: assuming invertibility satisfies the equations.}
Since we have a map in dimension 2, it is tame, and we can use the Jung-van der Kulk theorem. Since the degree is three or less, it is a map of the form $\alpha (x, y+f(x)) \beta $ where $\alpha, \beta $ are affine invertible maps, and $\deg(f)\leq 3$. (There can only be one triangular map involved, as the degree is prime.)
We can assume that $\beta=(ax+by+c, y)$ as we can put anything occuring in the second component in $f$.
Also, we can assume $f$ is of degree 2 or 3, and also we can assume that $f(0)=0$ as we can put any constant added in $\alpha$.
Adding in the requirement that the affine part of $\alpha (x, y+f(x)) \beta $ must be the identity yields requirements on $\alpha$ given $\beta$ and $f(x)$.
Working this out yields a generic map that is actually very similar to the formula of $T$ above; it can be easily checked that it satisfies the equations. \\

{\bf Remark:}
It is very hard to check this conjecture for specific degrees, even in $n=2$, as there is no shortcut other than doing hard-core computations. In
fact, en passant one is proving the conjecture, and the computations are very similar to proving the Jacobian Conjecture in characteristic zero - which is (we hope the reader agrees) a difficult task\ldots

Of course, we also checked the conjecture for many specific examples (which we do not list here, though we specifically mention that we could rule out ``obvious'' examples like $(x+x^p,y)$), though it feels a bit like touching a wall at random spots in the dark, not finding a light button and then shouting ``this wall has no light button''.

\section{Implications of the Jacobian Conjecture among various fields}

In this section we will see if $k,\acute{k}$ are two arbitrary fields of characteristics $p$ then what is the connection between $\mathcal{JC}(k,n)$ and $\mathcal{JC}(\acute{k},n)$ for all $n\geq1.$ We denote the Jacobian conjecture over all domains
having
 characteristic zero by $JC(n,0)$ and the Jacobian conjecture over all fields with characteristic $p$ by $\mathcal{JC}(n,p).$
In characteristics zero we have the following theorem (theorem 1.1.18 in \cite{E00}).
\begin{theorem}\label{char0}
 Let $R,\acute{R}$ be commutative rings contained in a $\Q$-algebra. If $JC(R,n)$ is true for all $n\geq1,$ then $JC(\acute{R},n)$ is true for all $n\geq1$.
\end{theorem}
  For the characteristic $p$ equivalent we have to assume part of our conjectures:

\begin{theorem}\label{charp} Assume the conjectures \ref{conj1}(2), \ref{conj2} are true. Let $k,\acute{k}$ be two fields contained in an $\F_p$-algebra such that $\mathcal{JC}(k,n)$ is true for all $n\geq1,$ then $\mathcal{JC}(\acute{k},n)$ is true for all $n\geq 1$. In particular, it is enough to verify $\mathcal{JC}(\F_p,n)$.
\end{theorem}
It is very hard to prove the statement \ref{charp} without making any assumption.
 To mention the main hurdle: consider $k$ is infinite field and $k_1\subset k$ a finite Galois extension, $a_1,a_2,\dots,a_n$ a $k_1$ basis of $k$ and denote by $\alpha:k_1^m\rightarrow k$ the map defined by $\alpha(y_1,\dots,y_n)=y_1a_1+\dots+y_na_n.$ The obvious extension $(\alpha,\dots,\alpha):(k_1^m)^n\lp k^n$ which we also denote by $\alpha,$ is clearly bijective. Let $F=(F_1,\dots,F_n):k^n\lp k^n$ be a polynomial map. So conjugate $F$ with $\alpha$ we get the map $F^{\alpha}:=\alpha^{-1}F\alpha:k_1^{mn}\lp k_1^{mn}.$ Comparing to the characteristics zero proof of theorem \ref{char0} there we know that $\det\Jac(F)\in k^*$ if and only if $\det\Jac(F^{\alpha})\in k_1^*$ (equation 1.1.26 in \cite{E00}). In characteristics $p$ we should have a similar statement that $F$ satisfies $J_{k}$ if and only if $F^{\alpha}$ satisfies $J_{k_1},$ but the proof of this is very difficult. This property that  $F$ satisfies $J_{k}$ if and only if $F^{\alpha}$ satisfies $J_{k_1}$ is needed to prove the theorem \ref{charp} if we don't assume that conjectures \ref{conj1}(2) and \ref{conj2} are true.
The remaining part of this section is devoted towards the proof of theorem \ref{charp}. We begin with some definitions and lemmas.

Let $\Omega:=\text{algebraic closure of }\F_p(\{x_i| i\in\N\})$ then $\Omega$ is a field with infinite transcendence degree over $\F_p.$
\begin{definition}\label{def}
Let $R,S$ be commutative rings.
Let $\phi:R\rightarrow S$ a ring homomorphism. If $F\in R[X]^n,$ then $F^{\phi}$ denotes the element of $S[X]^n$ obtained by applying $\phi$ to the coefficients of the $F_i$.
\end{definition}
We use the notation $X=(x_1,x_2,\ldots,x_n).$
The following proposition is taken from \cite{E00} (proposition 1.1.7 in \cite{E00}). $\eta$ is the nilradical of $R$.
\begin{proposition}(\textbf{Invertibility under base change})\label{invbc}
Let $\phi:R\rightarrow S$ be a ring homomorphism with $\ker\phi\subset\eta.$ Let $F\in R[X]^n$ with $\det JF(0)\in R^*.$ Then $F$ is invertible if and only if $F^\phi$ is invertible over $S.$
\end{proposition}
\begin{lemma} (\textbf{Embedding lemma})\label{embd}  Let $\F_p\subset \F_p(a_1,a_2,\dots,a_n)$ be a finitely generated field extension. Then there exists an isomorphism $\phi:\F_p(a_1,a_2,\dots,a_n)\simeq k\subset \Omega,$ where $k$ is a subfield of $\Omega.$
\end{lemma}
To prove this lemma we will use the following lemma which can be seen in any standard textbook on algebra (theorem 2.8 on page 233 in \cite{L02}).
\begin{lemma}\label{embd1} Let $K/k$ be an algebraic field extensions and let $\phi: k \to C$ be a ring homomorphism where $C$ is an algebraically closed field. Then there exists a ring homomorphism $\sigma : K \to C$ which extends $\phi.$
\end{lemma}

\begin{proof} (of embedding lemma)
 As $\F_p(a_1,a_2,\dots,a_n)$ is an algebraic extension of $\F_p$ and we have an inclusion $\F_p \to \Omega,$ so it has an extension $\F_p(a_1,a_2,\dots,a_n)\simeq k\subset \Omega$ by lemma \ref{embd1}.
\end{proof}

We can now use this to show that proving the $\mathcal{JC}$ for $\Omega$ is universal, in the sense that it proves the Jacobian Conjecture for all fields of the same characteristic.

\begin{proposition}
Let $n\geq1.$ If $\mathcal{JC}(\Omega,n)$ is true then $\mathcal{JC}(k,n)$ is true for any field $k$ of characteristics $p.$
\end{proposition}
\begin{proof}
Let $F\in k[X]^n$ satisfying $J_{\Z}\otimes k.$ Let $k_0$ be the subfield of $k$ generated over $\F_p$ by the coefficients of $F.$ Then $F$ satisfies $J_{\Z}\otimes k_0.$ By \ref{embd} we get an embedding $\phi:k_0\to \Omega.$ Since $F$ satisfies $J_{\Z}\otimes k_0$ we get that $F^{\phi}$ satisfies $J_{\Z}\otimes \phi(k_0)$ and hence by lemma \ref{rem1} $F^{\phi}$ satisfies $J_{\Z}\otimes \Omega.$ Hence $F^{\phi}$ is invertible over $\Omega$ since we assume that $JC(\Omega,n)$ is true.  So by \ref{invbc} $F$ is invertible over $k_0$ and hence over $k.$
\end{proof}
\begin{corollary}\label{countable}{of lemma \ref{rem1}\\} Let $n\geq1$ and $k_0\subset k$ be fields of characteristics $p.$ If $\mathcal{JC}(k,n)$ is true then $\mathcal{JC}(k_0,n)$ is true for any subfield $k_0$ of $k$.
\end{corollary}
\begin{proof}
Let $F\in\SKE_n(k_0)$ then by lemma \ref{rem1} we have $F\in\SKE_n(k).$ Since we assume that $JC(k,n)$ is true so $F$ is invertible over $k.$ Hence $F$ is invertible over $k_0$ by proposition \ref{invbc}.
\end{proof}
Let $k$ be any countable field of characteristics p. We can write $k=\{a_1,a_2,\dots\}$ where each $a_i\neq a_j$ for all $i\neq j.$
We can also assume that $k$ is ordered set. Corresponding to each element $a_i$ in $k$  consider the indeterminate $x_i.$ Define a polynomial ring over $\Z$ by $\Lambda_k:=\Z[x_1,x_2,...].$ Define a map by $\tau:\Lambda_k\rightarrow k$ by $x_i\mapsto a_i$ and $m\mapsto m\mod p$ for any $m\in\Z.$ Then it is clearly well defined surjective ring homomorphism.
Thus we have the following definition
 \begin{definition}\label{def1}
 For each countable field $k$ of characteristics $p$, define a polynomial ring $\Lambda_k$ over $\Z$ such that $\tau:\Lambda_k\lp k$ is surjective ring homomorphism.

 Notice that we can naturally extend $\tau$ to a map from  polynomial maps over $\Lambda_k$ to polynomial maps over $k$. We denote this extended map by $\pi$ as in definition \ref{def2}.
  \end{definition}

  Thus we have the following lemma.
\begin{lemma}\label{subset1}
Let $k$ be a countable field of characteristics $p.$ We have $\pi(\KE_n(\Lambda_k))\subseteq \SKE_n(k)$ for every $n\geq1.$
\end{lemma}
\begin{proof}
Let $F\in\KE_n(\Lambda_k),$ then $\det \Jac(F)=1$ and so $F$ satisfies $I_{\Q}$. Let $q\in J_k:=\bar{J}_{\Z}\otimes k$ then $q=\sum_{i}\tilde{e_i}h_i$ for $\tilde{e_i}\in \bar{J}_{\Z}$ and $h_i\in k$ for all $i$ (here $\tilde{e}_ih_i=\tilde{e}_i\otimes h_i,$ but we omit the tensor notation). Since $h_i\in k$ there exist $H_i\in \Lambda_k$ such that $\tau(H_i)=h_i.$ We can define a surjective homomorphism $J_{\Lambda_k}:=J_{\Z}\otimes {\Lambda_k}\lp \bar{J}_{\Z}\otimes k\text{ by }a\otimes b\mapsto \tau(a)\otimes \tau(b)$ where $a\in J_{\Z}$ and $b\in {\Lambda_k}.$
Thus there exist $Q\in J_{\Lambda_k}$ defined by $Q=\sum_{i}{e_i}H_i$ such that $q=\sum_{i}\tau(e_i) \tau(H_i)=\sum_{i}\tilde{e_i} h_i$ where $e_i\in J_{\Z}$ such that $\tau(e_i)=\tilde{e_i}$ for all $i$. By lemma \ref{ideals} we have $e_i(\nu(F))=0$ for all $i$ (since $F$ satisfies $I_{\Q}$).
   If we identify $x_i$ with $a_i$ as in definition of $\tau,$
    $\tilde{e_i}(\nu(\pi(F)))=e_i(\nu(F))\mod p=0\mod p$ for all $i.$ Thus $q(\nu(\pi(F)))=\sum_{i}\tilde{e_i}(\nu(\pi(F)))h_i(\nu(\pi(F)))=0\mod p.$ This shows that $\pi(F)$ satisfies $J_{k}.$ Hence $\pi(F)\in \SKE_n(k)$ which proves the lemma.

\end{proof}
Of course, the above lemma slightly reformulates conjecture \ref{conj2}:

\begin{corollary}\label{corsub1}
Assume  conjecture \ref{conj2} is true and $k$ be a countable field of characteristics $p.$ Then $\pi(\KE_n(\Lambda_k))=\SKE_n(k).$
\end{corollary}

We are now ready to link $JC(n,0)$ to $\mathcal{JC}(n,p)$.

\begin{proposition}\label{lemma1}\label{lemma2}~\\
(1)
Assume conjecture \ref{conj2} is true.
Then
\[ JC(n,0) \forall n\in \N^* \Longrightarrow \mathcal{JC}(n,p)\forall n\in \N^*.\]
(2)
Assume the conjectures \ref{conj1}(2), \ref{conj2} are true.
Then
\[ JC(n,0) \forall n\in \N^* \desda \mathcal{JC}(n,p)\forall n\in \N^*.\]
In fact, it is enough to prove or disprove $JC(n,\Z)$ for all $n$ to prove or disprove  $\mathcal{JC}(k,n)$ for all $n$ and for any field $k$.
\end{proposition}

\begin{proof}
(1)
Consider $K$ be an arbitrary field of characteristics $p$ and $f\in SK_n(K).$ Let $k$ be a subfield of $K$ generated over $\F_p$ by the coefficients of $f.$
 Since $k$ is at most countable, we have a surjective ring homomorphism $\tau:{\Lambda_k}\rightarrow k$ (definition \ref{def1}).
By corollary \ref{corsub1} there exists $F\in\KE_n({\Lambda_k})$ such that $\pi(F)=f$. Thus $F$ is invertible since we assume that $JC(n,0)$ is true, so there exist $G\in\ME_n(\Lambda_k)$ such that $F\circ G=I.$ Applying $\pi$ we have $\pi(F)\circ \pi(G)=\pi(I)=I\mod p.$ Thus $\pi(G)$ is an inverse of  $f=\pi(F).$ This shows that $f$ is invertible over $k$ and hence over $K.$\\
(2)
Let $K$ be an arbitrary field of characteristics $p$ and consider $k\subseteq K$ be a countable field (if K is itself countable then take $k=K$).
By corollary \ref{corsub1} we have $\pi(\KE_n({\Lambda_k}))=\SKE_n(k)$ (where $\Lambda_k$ is defined in \ref{def1}). Let $F\in\KE_n({\Lambda_k})$ then $\pi(F)$ satisfies $J_{k}.$
 Suppose $\mathcal{JC}(K,n)$ is true then $\mathcal{JC}(k,n)$ is true by corollary \ref{countable}.
  Thus $\pi(F)\in \SA_n(k)$ and so $F\in\pi^{-1}(\SA_n(k)).$ By conjecture \ref{conj1}(2) we have $F\in\SA_n({\Lambda_k}).$ By theorem \ref{char0} we have that $JC(n,0)$ is true.
\end{proof}
\begin{proof}(of theorem \ref{charp})\\
This is direct consequence of proposition \ref{lemma2}
\end{proof}

\section{Some results related to $\mathcal{JC}(k,n)$ }
In this section we present some basic results related to our formulation of the  Jacobian conjecture in characteristic $p$.
\subsection{Invertible polynomial maps and $\mathcal{JC}(k,n)$}
In this subsection we will discuss a natural question which can come to mind when studying the previous. If characteristic of $k$ is zero then we know that if $F\in\SA_n(k)$ then $F$ satisfies the keller condition $\det\Jac(F)=1$ (the only condition for Jacobian conjecture $JC(k,n)$). This is due to the fact that the determinant of the Jacobian has the property $\det\Jac(G\circ F)=\det\Jac(F)\cdot(\det\Jac(G)\circ F)$. If characteristic of $k$ is $p$ it is not easy to prove that if $F\in\SKE_n(k)$ then $F$ satisfies $J_k$ (the universal equations).

Nevertheless, assuming conjectures \ref{conj1}(1) and \ref{conj2} we can prove that if $F\in\SA_n(k)$ then it implies that $F\in\SKE_n(k).$
\begin{proposition}
Assume conjectures \ref{conj1}(1) and \ref{conj2} are true and $k$ be a field of characteristics $p.$ If $f\in\SA_n(k)$ then $f\in\SKE_n(k).$
\end{proposition}
\begin{proof}
Let $f\in\SA_n(k).$ Consider $k_0\subset k$ generated over $\F_p$ by the coefficients of $f.$  By conjecture \ref{conj1}(1) there exist some $F\in\SA_n(\Lambda_{k_0})$ such that $\pi(F)=f$ and thus $F$ satisfies $\KE_n(\Lambda_{k_0}).$
 Assuming conjecture \ref{conj2} we have $\pi(\KE_n(\Lambda_{k_0}))=\SKE_n(k_0)$ by corollary \ref{corsub1}. Thus $f\in \SKE_n(k_0)$ and hence $f\in\SKE_n(k)$ by lemma \ref{rem1}.
\end{proof}
\subsection{Closure property of $\SKE_n(k)$}

The set $\KE_n(R)$ is closed under composition for any ring $R$, and also for $R=k$ a field of characteristic $p$, even though it does not only consist of automorphisms. One would expect that $\SKE_n(\F_p)$ is also closed under composition. However, trying to prove this turns out to be an incredibly difficult task: if $F\in \SKE_n(\F_p)$ then the coefficients of $F$ satisfy certain equations that can be found in $J_p$. If we compose two such maps $F,G\in \SKE_n(\F_p)$, then the coefficients of the resulting map $F\circ G$ (denoted $v(F\circ G)$)  are polynomials in the coefficients of $F$ and $G$, i.e. $v(F\circ G)=P(v(F), v(G))$ for some polynomial map $P$.  To check if $F\circ G$ is in $\SKE_n(\F_p)$ we need to
see if  $v(F\circ G)$ satisfy (the equations in) $J_p$; however, this turned out to be extremely difficult.

Comparing to characteristic zero, there we know a priori due to the ``magical'' equation $\det\Jac(F\circ G) = \det\Jac(G)\cdot (\det\Jac(F)\circ G)$ that $\KE_n(\Z)$ is closed under composition. As a corollary, it gives that ``$v(F)$ satisfies $J_{\Z}$ and $v(G)$ satisfies $J_{\Z}$'' implies ``$v(F\circ G)$ satisfies $J_{\Z}$'', but exactly {\em how} is very complicated.

Nevertheless, making an assumption we can prove that $\SKE_n(\F_p)$ is closed under composition.

\begin{proposition}
Assume conjecture \ref{conj2} is true. Then $\SKE_n(k)$ is closed under composition, where $k$ be any field of characteristics $p$.
\end{proposition}
\begin{proof}
Let $f,g\in\SKE_n(k).$ Let $k_1$ be the subfield of $k$ generated over $\F_p$ by the coefficients of $f$ and $g$. Field $k_1$ is countable,
thus by corollary \ref{corsub1} there exist $F,G\in\KE_n(\Lambda_{k_1})$ such that $\pi(F)=f$ and $\pi(G)=g$. Now $F\circ G\in\KE_n(\Lambda_{k_1})$ as $\KE_n(\Lambda_{k_1})$ is closed under composition. Thus by corollary \ref{corsub1} $f\circ g=\pi(F)\circ \pi(G)=\pi(F\circ G)\in\SKE_n(k_1).$ Hence $f\circ g\in\SKE_n(k)$ (lemma \ref{rem1}).
\end{proof}

\subsection{Connections between $\mathcal{JC}(\F_p,n)$ and $JC(\Z,n)$.}

In this subsection we will see how we can move back and forth between $\mathcal{JC}(\F_p,n)$ and $JC(\Z,n).$
We quote theorem 10.3.13 from \cite{E00}. We will need this theorem to build the connection between $\mathcal{JC}(\F_p,n)$ and $JC(\Z,n)$.

\begin{theorem}\label{10.3.13}
Let $F\in\Z[x_1,x_2,\dots,x_n]^n.$ If $F\mod p:{\F_p}^n\rightarrow{\F_p}^n$ is injective for all but finitely many primes $p$ and $\det\Jac(F)\in\Z\setminus\{0\},$ then $F$ is invetible over $\Z.$
\end{theorem}

\begin{lemma}
If $\mathcal{JC}(\F_p,n)$ is true for all but finitely many primes $p$, then $JC(\Z,n)$ is true.
\end{lemma}

This is a slight variation on proposition \ref{lemma1} but without any requirements.

\begin{proof}
Let $F\in\MA_n(\Z),$ such that $\det(\Jac(F))=1.$ Then by lemma \ref{subset1} $F\mod p$ satisfies $J_p.$ Thus $F\mod p$ is invertible for almost all $p$ by given assumptions. By theorem \ref{10.3.13} we conclude that $F$ is invertible.
\end{proof}
For the converse of this lemma we need to assume conjecture \ref{conj2} to be true.
This in turn resembles proposition \ref{lemma2}.

\begin{lemma}
Suppose conjecture \ref{conj2} and $JC(\Z,n)$ are true, then $\mathcal{JC}(\F_p,n)$ is true.
\end{lemma}
\begin{proof}
Let $f\in\ME_n(\F_p)$ such that $f$ satisfies $J_p.$
By corollary \ref{corsub1} there exists $F\in\KE_n(\Z)$ such that $F\mod p=f$. By assumption, $F$ is invertible, so there exist $G\in\ME_n(\Z)$ such that $F\circ G=I.$ Thus $(F\mod p)\circ (G\mod p)=I\mod p$ and hence $G\mod p:=g$ is the inverse of  $f$.
\end{proof}

\subsection{Boundedness}

In this subsection we explore what happens if we assume that if the degree, or the degree and the  coefficents,  of a polynomial map is small with respect to $p$.
In some sense, the results say that if $p$ is ``small'' with respect to some formula depending on $p$ and $n$, then the situation is exactly the same as in characteristic zero.
We fix $n$ in this section, but note that the constant $N_d$ below depends also on $n$.

Let $\ME_n{({\F}_p)}^d$ be the set of polynomial endomorphisms of degree at most $d.$ Similarly we can define $\KE_n(\F_p)^d,$ $\SKE_n(\F_p)^d$ etc.
\begin{lemma}\label{local}
 Let $F\in ME_n{({\F}_p)}^d$ and $I_{Q}^d=(E_1,\dots,E_m)$ then there exist a positive integer $N_{d}$ such that for $p>N_d$ we have $J_{\Z_{(p)}}^d=\rad(E_1,\dots,E_m).$
 \end{lemma}

\begin{proof}
Consider the ideals $I_{\Q}^{d}=(E_1,\dots,E_m)$ and $I_{\Q}^d\cap C_{\Z_{(p)}}=(E_1,\dots,E_m,Q_1,\dots,Q_r),$ where  $Q_i=\frac{P_i(E_1,\dots,E_m)}{n_i},$ and $P_i(X)$ are polynomials with integer coefficients. Let $N_{d}=lcm(n_1,n_2,\dots,n_r).$ Then for $p>N_d$ we have $I_{\Q}^d\cap C_{R}=(E_1,\dots,E_m)$  where $R:=\Z[\frac{1}{N_{d}}].$  Hence $J_{\Z_{(p)}}^d:=\rad(I_{\Q}^{d}\cap C_{R})=\rad(E_1,\dots,E_m).$
\end{proof}

\begin{corollary}\label{pkeller} Let $F\in ME_n{({\F}_p)}^d$ and $I_{\Q}^d=(E_1,\dots,E_m)$ then there exist a positive integer $N_{d}$ such that for $p>N_d$ we have $J_{p}^d=\rad(E_1\mod p,\dots,E_m\mod p).$
\end{corollary}
\begin{proof}
Since by definition \[J_{p}^d=J_{\Z}^d\mod p=J_{\Z}^d\otimes_{\Z}\F_{p}\]\[=J_{\Z}^d\otimes_{\Z}(\Z_{(p)}\otimes_{\Z_{(p)}}\F_{p})\]\[=(J_{\Z}^d\otimes_{\Z}\Z_{(p)})\otimes_{\Z_{(p)}}\F_{p}\]
\[=J_{\Z_{(p)}}^d\otimes_{\Z_{(p)}}\F_{p}\]\[=J_{\Z_{(p)}}^d\mod p.\] By lemma \ref{local} we get $J_{p}^d=J_{\Z_{(p)}}^d\mod p=\rad(E_1\mod p,\dots,E_m\mod p).$
\end{proof}
\begin{corollary}\label{pkeller1}
There exist a positive integer $N_{d}$ such that $\KE_n(\Z)^{d}\mod p\subset\SKE_n(\F_p)^{d}$ for $p>N_{d}.$
\end{corollary}
\begin{proof} Direct consequence of corollary \ref{pkeller}.
\end{proof}

The following lemma is intuitively clear: if you have a polynomial map having coefficients which are (in $\Z$) small, then knowing that the map modulo $p$ is a (special) Keller map yields that it  was a Keller map to start with.

\begin{lemma}\label{pkC}
Let $f\in\SKE_n(\F_p)^d$ having coefficients bounded by some constant $C,$ (meaning here that for the coefficients a representative in $\Z$ can be picked in the interval $[-C,C]$). If $p$ is large enough with respect to $d$ and $C,$ then picking $F\in ME_n(\Z)^d$ such that $f=F\mod p$ and the coefficients of $F$ are in the interval $[-C,C]$, then $F\in \KE_n(\Z)^d.$
\end{lemma}
\begin{proof}
Consider ideals $I_{\Q}^d=(E_1,\dots,E_m)$ and $I_{\Q}^d\cap C_{\Z,d}=(Q_1,\dots,Q_r)$ such that $Q_i=\frac{P_i(E_1,\dots,E_m)}{n_i},$ where $P_i(X)$ are polynomials with integer coefficients for all $i$. Let $N_{d}=lcm(n_1,n_2,\dots,n_r)$ then for $p>N_d$ we have $I_{\Q}^d\cap C_{R,d}=(E_1,\dots,E_m)$ where $R=\Z[\frac{1}{N_{d}}].$
Let $f\in \SKE_n(\F_p)^d$ then $s(\nu(f))=0\mod p$ for all $s\in J_p^d.$ Consider $S\in J_{R}^d$ such that $S\mod p=s$ and $F\in \ME_n(\Z)$ such that $F\mod p=f$  then $S(\nu(F))\mod p=0\mod p$ for all $S\in J_{R}^d.$ Since $I_{\Q}^d\cap C_{R,d}\subset J_{R}^d$ thus for $p>N_d$ we have $E_i(\nu(F))\mod p=0\mod p$ for all $i.$ Define $N_i:=\max\{|E_i(\eta)|:\eta\in[C,C]^l, l=\text{the number of coefficients of the generic polynomial F} \}$ and $N_d(C):=\max\{N_d,N_1,N_2,\dots,N_m\},$ then for $p>N_d(C)$ we have $|E_i(\nu(F))|<p$ for all $1\leq i\leq m.$ Thus $E_i(\nu(F))=0$ for all $1\leq i\leq m$ for $p>N_d(C).$ Hence $F\in \KE_n(\Z)^d$ for $p>N_d(C).$
\end{proof}

Under some {\em very stringent conditions} we can now show closedness under composition of some elements in $\SKE_n(\F_p)$.
Let $\ME_n{({\F}_p)}^{d,C}$ be the set of polynomial endomorphisms of degree at most $d$ with bounded coefficients (indeed we can choose $C$ large enough and for bound coefficients are considered as representative in $\Z$). Similarly we can define $\KE_n(\F_p)^{d,C}$ $\SKE_n(\F_p)^{d,C}$ etc.

\begin{corollary}\label{A}
 There exist a positive integer $N_{d^2}(C)$ such that if $f,g\in\SKE_n(\F_p)^{d,C}$ with $p>N_{d^2}(C)$ then $f\circ g\in\SKE_n(\F_p)^{d^2,C}.$
\end{corollary}
\begin{proof}
Let $f,g\in\SKE_n(\F_p)^{d,C}$ such that for $F,G\in\ME_n(\Z),$  $F\mod p=f$ and $G\mod p=g.$ By lemma \ref{pkC} for $p>N_{d^2}(C)$ we have $F,G\in\KE_n(\Z)^{d,C}.$ Since $\KE_n(\Z)^{d,C}$ is closed under composition. Thus $F\circ G\in\KE_n(\Z)^{d^2,C}.$ Hence by corollary \ref{pkeller} for $p>N_{d^2}(C)$ we have $f\circ g\in\SKE_n(\F_p)^{d^2,C}.$
\end{proof}

The generic case eludes us:

\begin{conjecture} Let $k$ be a field of characteristic $p$. Then
$\SKE_n(k)$ is closed under composition.
\end{conjecture}

\end{document}